\documentclass[12pt,reqno]{amsart}
\usepackage[top=1.5in, bottom=1.5in, left=2.5in, right=2.5in]{geometry}
\usepackage[english]{babel}
\usepackage[utf8]{inputenc}
\usepackage[T1]{fontenc}
\usepackage{amsthm, amsmath, amssymb, amstext, amscd, amsfonts, amsbsy, amsxtra}
\usepackage{enumerate, enumitem, mathtools, mathrsfs, lmodern, microtype}
\usepackage{hyperref}

\setlist{nosep}
\setlist{noitemsep}
\numberwithin{equation}{section}

\theoremstyle{plain}
\newtheorem{thm}{Theorem}[section]
\newtheorem{cor}[thm]{Corollary}
\newtheorem{lem}[thm]{Lemma}

\newtheorem*{rem}{Remark}

\theoremstyle{definition}

\newtheorem{defn}[thm]{Definition}

\newcommand{\C}{\mathbb{C}}
\newcommand{\Hb}{\mathbb{H}}
\newcommand{\Z}{\mathbb{Z}}

\newcommand{\N}{\mathbb{N}}
\newcommand{\R}{\mathbb{R}}

\newcommand{\SL}{\mathrm{SL}}
\DeclareMathOperator{\re}{Re}

\DeclareMathOperator{\sgn}{sgn}
\newcommand{\vt}[1]{\left\lvert #1 \right\rvert}

\title{Flipping operators and locally harmonic Maass forms}
\author[K. Bringmann]{Kathrin Bringmann}
\address{University of Cologne, Department of Mathematics and Computer Science, Weyertal 86-90, 50931 Cologne, Germany}
\email{kbringma@math.uni-koeln.de}

\author[A. Mono]{Andreas Mono}
\author[L. Rolen]{Larry Rolen}
\address{Department of Mathematics, Vanderbilt University, 1326 Stevenson Center, Nashville, TN 37240, USA}
\email{andreas.mono@vanderbilt.edu}
\email{larry.rolen@vanderbilt.edu}

\date{\today}

\makeatletter
\@namedef{subjclassname@2020}{\textup{2020} Mathematics Subject Classification}
\makeatother

\begin{document}

\begin{abstract}
In the theory of integral weight harmonic Maass forms of manageable growth, two key differential operators, the Bol operator and the shadow operator, play a fundamental role. Harmonic Maass forms of manageable growth canonically split into two parts, and each operator controls one of these parts. A third operator, called the flipping operator, exchanges the role of these two parts. Maass--Poincar\'e series (of parabolic type) form a convenient basis of negative weight harmonic Maass forms of manageable growth, and flipping has the effect of negating an index. Recently, there has been much interest in locally harmonic Maass forms defined by the first author, Kane, and Kohnen. These are lifts of Poincar\'e series of hyperbolic type, and are intimately related to the Shimura and Shintani lifts. In this note, we prove that a similar property holds for the flipping operator applied to these Poincar\'e series.
\end{abstract}

\dedicatory{Dedicated to Krishnaswami Alladi, founding Editor-in-Chief of the Ramanujan Journal.}



\maketitle

\section{Introduction and statement of results}
Throughout, we let $k\in\N_{\ge2}$. The \emph{flipping operator} is defined as (see \cite{bkr})
\begin{equation*} 
	\mathfrak F_{2-2k}(f)(\tau) \coloneqq -\frac{v^{2k-2}}{(2k-2)!} \overline{R_{2-2k}^{2k-2}(f)(\tau)},
\end{equation*}
where $\tau=u+iv \in \Hb:=\{\tau \in \mathbb{C}:v>0\}$, the complex upper half-plane, throughout. Here, for $\kappa\in\Z$, the \emph{(iterated) Maass raising operator} is given as ($n \in \N$)
\begin{equation*}
R_{\kappa} \coloneqq 2i\frac d{d\tau} + \frac{\kappa}{v}, \quad R_{\kappa}^n \coloneqq R_{\kappa+2(n-1)} \circ \ldots \circ R_{\kappa}.
\end{equation*}
The goal of this paper is to show that the flipping operator keeps a certain quadratic form Poincar\'e series invariant up to sign.

The flipping operator acts on the weight $2-2k$ Maass--Poincar{\'e} series $\mathcal{P}_{2-2k,m}$ of index $m \in \Z \setminus \{0\}$ defined in \eqref{eq:MPdef} by (see \cite[Theorem 6.11 iv)]{BFOR})
\begin{align*}
\mathfrak F_{2-2k}(\mathcal{P}_{2-2k,m}) = \mathcal{P}_{2-2k,-m}.
\end{align*}
This ``flipping'' of $m$ and $-m$ is actually where the name ``flipping operator'' comes from.
For the Maass--Eisenstein series $\mathcal E_{2k}$ defined in \eqref{eq:MEdef}, we have (see \cite[Theorem 6.15 iv)]{BFOR})
\begin{align} \label{eq:Eisensteinflip}
\mathfrak F_{2-2k}(\mathcal E_{2k}) = -\mathcal E_{2k}.
\end{align}
The flipping operator is closely related to the {\it Bol operator} \cite{bronrh} and the {\it shadow operator} \cite{brufu}, which are respectively given by
	\begin{align*}
		\mathcal{D}^{2k-1} \coloneqq \left(\frac{1}{2\pi i} \frac\partial{\partial\tau} \right)^{2k-1}, \qquad \xi_{2-2k} \coloneqq 2iv^{2-2k} \overline{\frac\partial{\partial\overline{\tau}}}.
	\end{align*}
Let $P_{2k,m}$ and $E_{2k}$ be the usual (weakly) holomorphic Poincar{\'e} and Eisenstein series, respectively, defined in \eqref{eq:Pdef}. Then (see \cite[Theorem 6.11 ii), iii))]{BFOR}, we have
\begin{align*}
\xi_{2-2k} (\mathcal{P}_{2-2k,m}) = -\frac{(4\pi m)^{2k-1}}{(2k-2)!} P_{2k,-m}, \qquad \mathcal{D}^{2k-1} (\mathcal{P}_{2-2k,m}) = m^{2k-1} P_{2k,m}.
\end{align*}
For the Maass--Eisenstein series $\mathcal{E}_{2-2k}$ (see \cite[Theorem 6.15 ii), iii)]{BFOR}), we have
\begin{equation}\label{eq:parabolicimages}
\xi_{2-2k} (\mathcal E_{2-2k}) = (2k-1) E_{2k}, \quad \mathcal{D}^{2k-1} (\mathcal E_{2-2k}) = -(4\pi)^{1-2k}(2k-1)! E_{2k}.
\end{equation}
We refer to \eqref{E:XiFlip} and \eqref{E:DFlip} for the interplay between $\mathfrak{F}_{2-2k}$, $\xi_{2-2k}$, and $\mathcal{D}^{2k-1}$.

The aim of this paper is to show that \eqref{eq:Eisensteinflip} extends to a certain quadratic form Poincar\'e series. To describe this, we let $D \in \N$ be a non-square discriminant, and $\mathcal{Q}_D$ be the set of integral binary quadratic forms $Q(x,y) \coloneqq [a,b,c](x,y) \coloneqq ax^2+bxy+cy^2$ of discriminant $D \coloneqq b^2-4ac$ throughout. Zagier \cite{zagier75} defined
	\begin{equation*} 
		f_{k,D}(\tau) \coloneqq \frac{D^{k-\frac{1}{2}}}{\binom{2k-2}{k-1}\pi} \sum_{Q \in \mathcal Q_{D}} \frac{1}{Q(\tau,1)^k} \in S_{2k},
	\end{equation*}
where $S_{2k}$ denotes the space of weight $2k$ cusp forms for $\mathrm{SL}_2(\Z)$. These functions have very fruitful applications in the theory of modular forms \cite{kohnen85, koza81, koza84}. Katok \cite{katok} showed that one can write $f_{k,D}$ as a hyperbolic Eisenstein series (see \cite{IOS} for more details on hyperbolic expansions), and used this to prove that the $f_{k,D}$ generate $S_{2k}$ as $D$ ranges over positive discriminants. Paralleling the behavior of the parabolic Eisenstein series in \eqref{eq:parabolicimages}, the first author, Kane, and Kohnen \cite{BKK} discovered a weight $2-2k$ preimage $\mathcal F_{1-k,D}$ of $f_{k,D}$ (up to constants) under both the Bol operator and the shadow operator. To introduce it, we define for $Q=[a,b,c] \in \mathcal{Q}_D$
\begin{align*}
Q_{\tau} \coloneqq \frac{1}{v}\left(a\left(u^2+v^2\right)+bu+c\right), \qquad \beta(x;r,s) \coloneqq \int_{0}^{x} t^{r-1} (1-t)^{s-1} dt
\end{align*}
for $x \in [0,1]$ and $r,s > 0$. Then the function $\mathcal F_{1-k,D}$ is given by (see \cite[(1.4)]{BKK})
	\begin{equation} \label{eq:Fcdef}
		\mathcal F_{1-k,D}(\tau) \!\coloneqq\! \frac{D^{\frac{1}{2}-k}}{2\binom{2k-2}{k-1}\pi} \sum_{Q \in \mathcal Q_D} \sgn(Q_{\tau})Q(\tau,1)^{k-1} \beta\!\left(\frac{Dv^2}{\vt{Q(\tau,1)}^2}; k\!-\!\frac{1}{2}, \frac{1}{2}\right)\!.
	\end{equation}
Analogously to $f_{k,D}$, one may view $\mathcal F_{1-k,D}$ as a hyperbolic Eisenstein series of negative weight. Paralleling \eqref{eq:parabolicimages}, it was shown in \cite[Theorem 1.2]{BKK},
	\begin{equation}\label{E:DF}
		\xi_{2-2k}(\mathcal F_{1-k,D}) = D^{\frac12-k} f_{k,D}, \quad \mathcal D^{2k-1}(\mathcal F_{1-k,D}) = -\frac{(2k-2)!}{(4\pi)^{2k-1}} D^{\frac12-k} f_{k,D}
	\end{equation}
outside a certain ``exceptional set'' $E_D \subseteq \Hb$ (see \eqref{eq:EDdef}) of measure\footnote{The set $E_D$ is a union of geodesics associated to $Q \in \mathcal{Q}_D$.} $0$. Thus, it is natural to ask if \eqref{eq:Eisensteinflip} extends from a parabolic to a hyperbolic setting as well. We show that this is indeed the case.
	\begin{thm}\label{T:Main}
		If $D \in \N$ is a non-square discriminant, $k\ge2$, and $\tau \not\in E_D$, then we have
		\begin{equation*}
			\mathfrak F_{2-2k}(\mathcal F_{1-k,D}(\tau)) = -\mathcal F_{1-k,D}(\tau).
		\end{equation*}
	\end{thm}

\begin{rem}
In \cite{egkr, mmrw}, the function $\mathcal F_{1-k,D}$ played a crucial role to characterize non-trivial vanishing of twisted central $L$-values. More precisely, the main result of both papers is that a twisted central $L$-value of a newform vanishes if and only if an associated ``local polynomial'' (see \eqref{eq:localPolydef}) is constant. Theorem \ref{T:Main} might help to simplify the methods developed in \cite{egkr, mmrw} by detecting the vanishing of the non-constant part of that local polynomial.
\end{rem}

In \cite{brimo1}, the first two authors introduced and investigated the weight $-2k$ variant
	\begin{equation*} 
		\mathcal G_{-k,D}(\tau) \!\coloneqq\! \frac{1}{2} \sum_{Q \in \mathcal Q_D} Q(\tau,1)^{k} \beta\!\left(\frac{Dv^2}{\vt{Q(\tau,1)}^2}; k\!+\!\frac{1}{2}, \frac{1}{2}\right)\!
	\end{equation*}
of $\mathcal{F}_{1-k,D}$, which has continuously but not differentially removable singularities on $E_D$ (see Lemma \ref{lem:brimoresults}). One may view $\mathcal G_{-k,D}$ as an ``even'' analog of $\mathcal F_{1-k,D}$. Hence, it is natural to expect that Theorem \ref{T:Main} extends to $\mathcal G_{-k,D}$. This is indeed the case, as the following corollary shows.
\begin{cor} \label{cor:Maincor}
If $D \in \N$ is a non-square discriminant, $k\ge2$, and $\tau \not\in E_D$, then we have
\begin{align*}
\mathfrak F_{-2k}(\mathcal G_{-k,D}(\tau)) = -\mathcal G_{-k,D}(\tau).
\end{align*}
\end{cor}

The paper is organized as follows. In Section \ref{Prelim} we review the necessary preliminaries on modular forms, harmonic Maass forms, and locally harmonic Maass forms. Section~\ref{ProofMain} is devoted to the proofs of Theorem \ref{T:Main} and of Corollary \ref{cor:Maincor}.

\section*{Acknowledgments}
We thank the referee for helpful comments. The first author was funded by the European Research Council (ERC) under the European Union’s Horizon 2020 research and innovation programme (grant agreement No. 101001179). The second author received support for this research by an AMS--Simons Travel Grant. The third author's work was supported by a grant from the Simons Foundation (853830, LR).

\section{Preliminaries}\label{Prelim}
We summarize some standard notions from the theory of holomorphic modular forms and harmonic Maass forms. We refer to \cite{BFOR} for more details on such forms. More details on locally harmonic Maass forms can be found in \cite{BKK}.

\subsection{Poincar{\'e} series}\label{Poincare}
For a function $f\colon \mathbb H\rightarrow\C$, $\kappa\in\Z$, and a matrix $\begin{psmallmatrix} a & b \\ c & d\end{psmallmatrix}\in \Gamma \coloneqq \SL_2(\Z)$, define the \emph{Petersson slash operator}
\begin{align*}
f\big\vert_{\kappa}\begin{pmatrix} a & b \\ c & d\end{pmatrix}(\tau) \coloneqq (c\tau+d)^{-\kappa} f\left(\frac{a\tau+b}{c\tau+d}\right).
\end{align*}
For  completeness, we define the variants of modular forms we need.
\begin{defn}
Let $f \colon \Hb \to \C$ be a function and $\kappa \in \Z$.
\begin{enumerate}[label=\rm(\arabic*),leftmargin=*]
\item We say $f$ is a {\it holomorphic modular form} of weight $\kappa$ if the following hold:
\begin{enumerate}[label=\rm(\roman*),leftmargin=0.8cm]
\item For $\gamma \in \Gamma$ we have $f\vert_{\kappa}\gamma = f$,
\item $f$ is holomorphic on $\Hb$,
\item $f$ is holomorphic at $i\infty$.
\end{enumerate}
We denote the vector space of functions satisfying these conditions by $M_{\kappa}$.
\item If in addition $f$ vanishes at $i\infty$, then we call $f$ a \textit{cusp form}. The space of cusp forms is denoted by $S_{\kappa}$.
\item If $f$ satisfies the conditions (i) and (ii) from (1) and is permitted to have a pole at $i\infty$, then we call $f$ a {\it weakly holomorphic modular form} of weight ${\kappa}$. The vector space of such functions is denoted by $M_{\kappa}^!$.
\end{enumerate}
\end{defn}

To define examples, for $k\in\N_{\ge2}$ and $m \in \Z$, set
\begin{align} \label{eq:Pdef}
P_{2k,m} \coloneqq \sum_{\gamma \in  \Gamma_{\infty} \backslash \Gamma} \varphi_m \big\vert_{2k}\gamma, \qquad E_{2k} \coloneqq P_{2k,0},
\end{align}
where $\varphi_m(\tau)\coloneqq q^m$ with $q \coloneqq e^{2\pi i \tau}$ and $\Gamma_{\infty} \coloneqq \left\{\pm\begin{psmallmatrix} 1 & n \\ 0 & 1 \end{psmallmatrix} \colon n \in \Z \right\}$. We have
\begin{align*}
P_{2k,m} \in \begin{cases}
S_{2k} &\text{if } m > 0, \\
M_{2k} &\text{if } m = 0, \\
M_{2k}^! &\text{if } m < 0,
\end{cases}
\end{align*}
and $P_{2k,m}$ spans $S_{2k}$ resp. $M_{2k}^!$ for $m \neq 0$ (see \cite[Theorems 6.8 and 6.9]{BFOR}).

In negative weights, we require the following non-holomorphic modular forms.
\begin{defn}
Let $f \colon \Hb \to \C$ be a smooth function.
\begin{enumerate}[label=\rm(\arabic*),leftmargin=*]
\item We say $f$ is a weight $2-2k$ {\it harmonic Maass form} if the following hold:
\begin{enumerate}[label=\rm(\roman*),leftmargin=0.8cm]
	\item For $\gamma \in \Gamma$ we have $f\vert_{2-2k}\gamma = f$.
	\item We have
	$
		\Delta_{2-2k}(f) = 0,
	$
	where the weight $\kappa$ {\it hyperbolic Laplacian} is defined as
	\begin{align*}
	\Delta_{2-2k} \coloneqq -v^2\left(\frac{\partial^2}{\partial u^2}+\frac{\partial^2}{\partial v^2}\right) + i(2-2k) v\left(\frac{\partial}{\partial u} + i\frac{\partial}{\partial v}\right).
	\end{align*}
	\item There exists a polynomial $P_f(\tau)\in\C[q^{-1}]$, the \emph{principal part of $f$} such that, as $v \to \infty$,
\begin{align*}
f(\tau) - P_f(\tau) = O\left(e^{-\varepsilon v}\right)
\end{align*}
for some $\varepsilon > 0$.
\end{enumerate}
We denote the vector space of such functions by $H_{\kappa}$.
\item If $f$ satisfies the conditions (i) and (ii) from (1) as well as
\begin{align*}
f(\tau) = O\left(e^{\varepsilon v}\right) \ \text{as} \ v \to \infty \ \text{for some} \ \varepsilon > 0,
\end{align*}
then we call $f$ a {\it harmonic Maass form of manageable growth}. We denote the vector space of such forms by $H_{\kappa}^!$.
\end{enumerate}
\end{defn}

Let $M_{\mu,\nu}$ be the $M$-Whittaker function (see \cite[Subsection 13.14]{nist} for example) and define the \emph{Maass--Poincar{\'e} series} (see \cite[Definition 6.10]{BFOR})
\begin{align} \label{eq:MPdef}
\mathcal{P}_{2-2k,m} \coloneqq \sum_{\gamma \in \Gamma_{\infty} \backslash \Gamma} \phi_{2-2k,m}\big\vert_{2-2k} \gamma, \qquad m \in \Z \setminus \{0\},
\end{align}
where
\begin{align*}
\phi_{2-2k,m}(\tau) \coloneqq
\frac{(-\sgn(m))^{2k-1} (4 \pi |m| v)^{k-1}}{(2k-1)!} M_{\sgn(m)(1-k),k-\frac12}(4 \pi |m| v) e^{2 \pi i m u}.
\end{align*}
The {\it Maass--Eisenstein series} is given by (see \cite[p.\ 104]{BFOR})
\begin{align} \label{eq:MEdef}
\mathcal E_{2-2k} \coloneqq \sum_{\gamma \in \Gamma_{\infty} \backslash \Gamma} v^{2k-1}\big\vert_{2-2k} \gamma.
\end{align}
We have $\mathcal E_{2-2k} \in H_{2-2k}^!$ and
\begin{align*}
\mathcal{P}_{2-2k,m} \in \begin{cases}
H_{2-2k}^! &\text{if } m > 0, \\
H_{2-2k} &\text{if } m < 0.
\end{cases}
\end{align*}
By virtue of their Fourier expansion (see \cite[Theorem 6.11 v)]{BFOR}), $\mathcal{P}_{2-2k,m}$ has a prescribed principal part for $m < 0$ resp.\ a prescribed non-holomorphic part for $m > 0$ (see \eqref{eq:fouriershape} resp.\ \eqref{eq:fouriershape2} below). For this reason, $\mathcal{P}_{2-2k,m}$ span $H_{2-2k}$ (resp.\ $H_{2-2k}^!$ if restricting to $m > 0$). We refer the reader to \cite[Theorems 6.11 and 6.15]{BFOR} for proofs including \eqref{eq:Eisensteinflip}, \eqref{eq:parabolicimages} and more details.

\subsection{Differential operators and Fourier expansions}\label{FourierDifferential}
Let $\gamma \in \SL_2(\R)$ and $f \colon \Hb \to \C$ be smooth. Then, the Maass raising operator satisfies (see \cite[Lemma 2.1.1]{B})
\begin{equation} \label{eq:maassraisingslashing}
R_{\kappa} \left(f \vert_{\kappa} \gamma\right) = R_{\kappa}(f) \vert_{\kappa+2}\gamma.
\end{equation}
Moreover, if $f$ is an eigenfunction of $\Delta_{\kappa}$ with eigenvalue $\lambda$, then $R_{\kappa}(f)$ is an eigenfunction of $\Delta_{\kappa+2}$ with eigenvalue $\lambda+\kappa$, see \cite[Lemma 5.2]{BFOR}. Furthermore, according to \cite[(56)]{123}, the iterated Maass raising operator and the Bol operator are related by
	\begin{equation}\label{E:RD}
		R_{2-2k}^n = \sum_{r=0}^n (-1)^r \binom nr (2-2k+r)_{n-r} v^{r-n} (4\pi)^r \mathcal D^r,
	\end{equation}
	where the {\it rising factorial} is defined as
\begin{align*}
(a)_n\coloneqq a(a+1)\cdots(a+n-1), \qquad n \in \N_0.
\end{align*}

Bruinier and Funke \cite{brufu} showed that the Fourier expansion of a harmonic Maass form of manageable growth\footnote{The reader should be aware that their terminology refers to our harmonic Maass forms of manageable growth as ``\textit{weak} Maass forms.''} $f \in H_{2-2k}^{!}$ naturally splits into a holomorphic part and a non-holomorphic part. Namely (see \cite[Lemma 4.3]{BFOR}), we have a Fourier expansion of the shape
\begin{align} \label{eq:fouriershape}
f(\tau) = \sum_{n \gg -\infty} c_f^+(n) q^n + c_f^-(0)v^{2k-1} + \sum_{\substack{n \ll \infty \\ n \neq 0}} c_f^-(n)\Gamma(2k-1,-4\pi nv)q^n.
\end{align}
Here, the non-holomorphic part involves the {\it incomplete Gamma function}
\begin{align*}
\Gamma(s,x) \coloneqq \int_x^{\infty} t^{s-1}e^{-t} dt,
\end{align*}
defined for $\re(s) > 0$ and $x \in \R$. In particular, if $f \in H_{2-2k}$, then $f$ has a Fourier expansion of the shape
\begin{align} \label{eq:fouriershape2}
f(\tau) = \sum_{n \gg -\infty} c_f^+(n) q^n + \sum_{n < 0} c_f^-(n)\Gamma(2k-1,-4\pi nv)q^n.
\end{align}
By \cite[Proposition 5.15 iii), iv)]{BFOR}, the flipping operator satisfies
	\begin{align}
		\xi_{2-2k}(\mathfrak F_{2-2k}(f)) &= \frac{(4\pi)^{2k-1}}{(2k-2)!} \mathcal D^{2k-1}(f),\label{E:XiFlip}\\
		\mathcal D^{2k-1}(\mathfrak F_{2-2k}(f)) &= \frac{(2k-2)!}{(4\pi)^{2k-1}} \xi_{2-2k}(f). \label{E:DFlip}
	\end{align}
In other words, the flipping operator ``switches'' the holomorphic and non-holomorphic part in the Fourier expansion of a harmonic Maass form of manageable growth.

\subsection{Locally harmonic Maass forms}\label{FkDProp}
Let $D \in \N$ be a non-square discriminant and define
\begin{align} \label{eq:EDdef}
E_D \coloneqq \bigcup_{Q \in \mathcal{Q}_D} \left\{\tau \in \Hb \colon Q_{\tau} = 0 \right\}.
\end{align}
We recall the following definition from \cite[Section 2]{BKK}.
\begin{defn} \label{defn:LHMF}
	A function $f\colon \Hb \to\C$ is called a \emph{locally harmonic Maass form of weight $2-2k$} with exceptional set $E_D$, if it satisfies the following conditions:
	\begin{enumerate}[leftmargin=*,label=(\arabic*)]
		\item We have $f\vert_{2-2k}\gamma=f$ for every $\gamma \in \Gamma$.

		\item For all $\tau\in\Hb\setminus E_D$, there exists a neighborhood of $\tau$, in which $f$ is real-analytic and $\Delta_{2-2k}(f)(\tau)=0$.

		\item For every $\tau\in E_D$, we have that
		\[
			f(\tau) = \frac12\lim_{\varepsilon\to0^+} (f(\tau+i\varepsilon)+f(\tau-i\varepsilon)).
		\]

		\item The function $f$ exhibits at most polynomial growth towards $i\infty$.
	\end{enumerate}
\end{defn}

According to \cite[Theorem 1.1]{BKK}, $\mathcal{F}_{1-k,D}$ from \eqref{eq:Fcdef} is such a locally harmonic Maass form with exceptional set $E_D$. A key property of $\mathcal F_{1-k,D}$ is that it admits a certain splitting, which might be viewed as a generalization of the Fourier expansion \eqref{eq:fouriershape}. To describe this, we define the \emph{holomorphic} and \emph{non-holomorphic Eichler integrals} \cite[Section 2]{eichler} of a cusp form $f(\tau)=\sum_{n\ge1} c_f(n) q^n\in S_{2k}$ by
\begin{align*}
	\mathcal E_f(\tau) \coloneqq \sum_{n\ge1} \frac{c_f(n)}{n^{2k-1}}q^n, \qquad
	f^*(\tau) \coloneqq (2i)^{1-2k}\!\int_{-\overline\tau}^{i\infty} f^c(z)(z+\tau)^{2k-2} dz,
	\end{align*}
where $f^c(\tau)\coloneqq\overline{f(-\overline \tau)}$. We have
\begin{equation*}
\xi_{2-2k} \left(f^{*}\right) = f, \quad \mathcal{D}^{2k-1} \left(f^{*}\right) = 0, \quad \xi_{2-2k} \left(\mathcal{E}_{f}\right) = 0,  \quad \mathcal{D}^{2k-1} \left(\mathcal{E}_{f}\right) = f.
\end{equation*}
Note that both $\mathcal E_f$ and $f^*$ are real-analytic on $\Hb$.

\begin{lem}[\protect{\cite[Theorem 7.1]{BKK}}]\label{L:LocPoly}
Let $\mathcal{C} \subseteq \Hb \setminus E_D$ be a connected component and $\tau\in\mathcal C$. Let
\begin{equation*} 
c_\infty(k) \coloneqq \frac{1}{2^{2k-2}(2k-1)} \sum_{a\ge1} \sum_{\substack{0\le b<2a\\b^2\equiv D\pmod{4a}}} \frac{1}{a^{k}},
\end{equation*}
and define the local polynomial by
\begin{equation} \label{eq:localPolydef}
P_{\mathcal{C}}(\tau) \coloneqq -\frac{c_\infty(k)}{\binom{2k-2}{k-1}} + (-1)^k 2^{3-2k} D^{\frac12-k} \sum_{\substack{Q=[a,b,c]\in\mathcal Q_D\\a<0<Q_\tau}} Q(\tau,1)^{k-1}.
\end{equation}
Then, we have
	\begin{equation*}
		\mathcal F_{1-k,D}(\tau) = D^{\frac12-k} f_{k,D}^*(\tau) - D^{\frac12-k} \frac{(2k-2)!}{(4\pi)^{2k-1}} \mathcal E_{f_{k,D}}(\tau) + P_{\mathcal{C}}(\tau).
	\end{equation*}
\end{lem}

\begin{rem}
The constant $c_\infty(k)$ can be evaluated using \cite[Proposition 3]{zagier77}.
\end{rem}

Note that the splitting in Lemma \ref{L:LocPoly} resembles the Fourier expansion \eqref{eq:fouriershape} of a harmonic Maass form of manageable growth with $c_f^-(0)v^{k-1}$ replaced by the local polynomial $P_{\mathcal{C}}$.

Inspecting a modular integral introduced by Parson \cite{parson} in $1993$, the second author \cite{mo1} introduced the function
\begin{align*}
g_{k+1,D}(\tau) \coloneqq \sum_{Q \in \mathcal{Q}} \frac{\sgn\left(Q_{\tau}\right)}{Q(\tau,1)^{k+1}}, \qquad \tau \not\in E_D,
\end{align*}
which can be thought of as an ``odd'' local variant of $f_{k,D}$ as well as a positive weight analog of $\mathcal{F}_{1-k,D}$, because it exhibits singularities on the set $E_D$ too. The function $g_{k+1,D}$ satisfies the conditions from Definition \ref{defn:LHMF} by \cite[Theorem 1.1]{mo1}. Moreover, this function appeared in \cite{brimo1} while investigating certain functions introduced by Knopp \cite{knopp} in $1990$. We conclude by citing the following results from \cite{brimo1}.
\begin{lem} \label{lem:brimoresults}
\
\begin{enumerate}[leftmargin=*]
\item[\rm(1)] Both Eichler integrals $\mathcal{E}_{g_{k+1,D}}$ and $g_{k+1,D}^{*}$ may be defined on $E_D$.
\item[\rm(2)] The function $\mathcal{G}_{-k,D}$ is a locally harmonic Maass form of weight $-2k$ with continuously but not differentially removable singularities on $E_D$. 
\item[\rm(3)] If $\tau \in \Hb \setminus E_D$, then we have
\begin{align*}
\hspace*{\leftmargini} \mathcal{G}_{-k,D}(\tau) = \pi D^{k+\frac{1}{2}}c_{\infty}(k+1) - \frac{D^{k+\frac{1}{2}}(2k)!}{(4\pi)^{2k+1}}\mathcal{E}_{g_{k+1,D}}(\tau) + D^{k+\frac{1}{2}} g_{k+1,D}^{*}(\tau).
\end{align*}
In particular, we have
\begin{align*}
\mathcal{D}^{2k+1}\left(\mathcal{G}_{-k,D}(\tau)\right) &= - \frac{D^{k+\frac{1}{2}}(2k)!}{(4\pi)^{2k+1}}g_{k+1,D}(\tau), \\ 
\xi_{-2k}\left(\mathcal{G}_{-k,D}(\tau)\right) &= D^{k+\frac{1}{2}}g_{k+1,D}(\tau).
\end{align*}
\end{enumerate}
\end{lem}

\begin{proof}
The first part is the remark after \cite[Proposition 4.4]{brimo1}. The second part is \cite[Theorem 1.3 (1)]{brimo1}. The third part is \cite[Theorem 1.3 (2)]{brimo1} and \cite[Proposition 5.2 (1) and (2)]{brimo1}.
\end{proof}

\section{Proof of Theorem \ref{T:Main} and Corollary \ref{cor:Maincor}}\label{ProofMain}
In this section we prove our main result and our corollary.
	\begin{proof}[Proof of Theorem \ref{T:Main}]
		We use Lemma \ref{L:LocPoly} and define the real-analytic function
		\begin{equation*}
			\mathbb F \coloneqq D^{\frac12-k} f_{k,D}^* - D^{\frac12-k} \frac{(2k-2)!}{(4\pi)^{2k-1}} \mathcal E_{f_{k,D}}.
		\end{equation*}
		Using \eqref{E:XiFlip}, \eqref{E:DFlip}, and \eqref{E:DF} gives that
		\begin{equation*}
			\xi_{2-2k}(\mathfrak F_{2-2k}(\mathbb F)) = -\xi_{2-2k}(\mathbb F), \quad \mathcal{D}^{2k-1}(\mathfrak F_{2-2k}(\mathbb F)) = -\mathcal{D}^{2k-1}(\mathbb F).
		\end{equation*}
		Since both terms of $\mathbb F$ decay towards $i\infty$ ($f_{k,D}$ is a cusp form), this shows that
		\begin{equation*}
			\mathfrak F_{2-2k}(\mathbb F) = -\mathbb F.
		\end{equation*}

		So we are left to prove that $\mathfrak F_{2-2k}(P_{\mathcal{C}})=-P_{\mathcal{C}}$. By definition of $P_{\mathcal{C}}$ in \eqref{eq:localPolydef}, it suffices to show that
		\begin{align}
		\mathfrak F_{2-2k}(1) &= -1,\label{eq:claimPart1}\\
		\mathfrak F_{2-2k} (p) &= -p,\label{eq:claimPart2}
		\end{align}
		where
		\begin{equation*}
			p(\tau) \coloneqq (-1)^k 2^{3-2k} D^{\frac12-k} \sum_{\substack{Q=[a,b,c]\in\mathcal Q_D\\a<0<Q_\tau}} Q(\tau,1)^{k-1}.
		\end{equation*}

We start by proving \eqref{eq:claimPart1}, which is equivalent to showing that
\begin{equation*}
R_{2-2k}^{2k-2}(1) = (2k-2)! v^{2-2k}.
\end{equation*}
This follows directly using \eqref{E:RD}.

We next prove \eqref{eq:claimPart2}. Since $\tau \not\in E_D$ and $D$ a non-square, we may rewrite
\begin{align*}
p(\tau) = (-1)^k 2^{1-2k} D^{\frac12-k} \sum_{Q=[a,b,c]\in\mathcal Q_D} \left(1-\sgn(a)\right)\left(\sgn\left(Q_\tau\right)+1\right)Q(\tau,1)^{k-1}.
\end{align*}
Next observe that
\begin{equation*}
R_{\kappa} \left(\left(\sgn\left(Q_\tau\right)+1\right)Q(\tau,1)^{k-1}\right) = \left(\sgn\left(Q_\tau\right)+1\right) R_{\kappa} \left(Q(\tau,1)^{k-1}\right)
\end{equation*}
for $\kappa \in \Z$, because $\frac{d}{d\tau} (\sgn(Q_\tau)+1) = 0$. Thus, \eqref{eq:claimPart2} follows if we show that
\begin{align*}
\mathfrak F_{2-2k}\!\left(Q(\tau,1)^{k-1}\right) = -Q(\tau,1)^{k-1}.
\end{align*}
This is equivalent to
\begin{align} \label{eq:claimPart2version2}
R_{2-2k}^{2k-2}\!\left(Q(\tau,1)^{k-1}\right) = \frac{(2k-2)!}{v^{2k-2}} Q(\overline\tau,1)^{k-1}.
\end{align}
To prove this  note that there exists $A\in\SL_2(\R)$ (see \cite[Lemma 3.1]{BKK}) such that
\begin{equation} \label{eq:Qslash}
	\tau \vert_{-2} A = -\frac{Q(\tau,1)}{\sqrt D}.
\end{equation}
By the definition of the slash operator, \eqref{eq:Qslash} is equivalent to
\begin{align*}
\tau^{k-1} \big|_{2-2k} A = \frac{(-1)^{k+1}}{D^\frac{k-1}2} Q(\tau,1)^{k-1}.
\end{align*}
Using \eqref{eq:maassraisingslashing} repeatedly, we infer
		\begin{equation*}
			R_{2-2k}^{2k-2}\left(Q(\tau,1)^{k-1}\right) = (-1)^{k+1} D^\frac{k-1}2 R_{2-2k}^{2k-2}\left(\tau^{k-1}\right)\Big|_{2k-2}A.
		\end{equation*}
To prove \eqref{eq:claimPart2version2}, and hence \eqref{eq:claimPart2} as well, we next claim that it suffices to verify
		\begin{equation}\label{E:RaiseTau}
			R_{2-2k}^{2k-2}\left(\tau^{k-1}\right) = (2k-2)! \left(\frac{\overline\tau}{v^2}\right)^{k-1}.
		\end{equation}
		Indeed, assuming \eqref{E:RaiseTau} gives
		\begin{equation*}
			R_{2-2k}^{2k-2}\left(Q(\tau,1)^{k-1}\right) = (-1)^{k+1} D^\frac{k-1}2 (2k-2)! \left(\frac{\overline\tau^{k-1}}{v^{2k-2}}\right) \bigg|_{2k-2} A.
		\end{equation*}
		Rewriting yields
		\begin{equation*}
			\left(\frac{\overline\tau^{k-1}}{v^{2k-2}}\right) \bigg|_{2k-2} A = \frac{(-1)^{k+1}}{D^\frac{k-1}2} v^{2-2k} Q(\overline\tau,1)^{k-1}.
		\end{equation*}
		This directly gives \eqref{eq:claimPart2version2}.

		Hence, we are left to prove \eqref{E:RaiseTau}. By \eqref{E:RD}, we have
		\begin{equation*}
			R_{2-2k}^{2k-2}\left(\tau^{k-1}\right) = \sum_{r=0}^{2k-2} (-1)^r \binom{2k-2}r (2-2k+r)_{2k-2-r} v^{r-2k+2} (4\pi)^r \mathcal D^r\left(\tau^{k-1}\right).
		\end{equation*}
		If $r>k-1$, then $\mathcal D^r(\tau^{k-1}) = 0$. If $r\le k-1$, then
		\begin{equation*}
			\mathcal D^r\left(\tau^{k-1}\right) = \frac1{(2\pi i)^r} \frac{(k-1)!}{(k-r-1)!} \tau^{k-1-r}.
		\end{equation*}
		Moreover
		\begin{equation*}
			(2-2k+r)_{2k-2-r} = (-1)^r (2k-r-2)!.
		\end{equation*}
		 We then verify \eqref{E:RaiseTau} by obtaining
		\begin{align*}
			R_{2-2k}^{2k-2}\!\left(\tau^{k-1}\right) &= (2k\!-\!2)! \!\left(\frac{\tau}{v^2}\right)^{\!k-1} \sum_{r=0}^{k-1} \frac{(k\!-\!1)!}{r!(k\!-\!1\!-\!r)!} \!\left(\frac{-2iv}\tau\right)^{\!r}
			= (2k\!-\!2)! \!\left(\frac{\overline\tau}{v^2}\right)^{\!k-1}.
		\end{align*}
This completes the proof.
	\end{proof}

It remains to prove Corollary \ref{cor:Maincor}.
\begin{proof}[Proof of Corollary \ref{cor:Maincor}]
Let $\tau \in \Hb \setminus E_D$ and define
\begin{align*}
\mathbb{G}(\tau) \coloneqq D^{k+\frac{1}{2}} g_{k+1,D}^{*}(\tau) - \frac{D^{k+\frac{1}{2}}(2k)!}{(4\pi)^{2k+1}}\mathcal{E}_{g_{k+1,D}}(\tau).
\end{align*}
Using \eqref{E:XiFlip}, \eqref{E:DFlip}, and Lemma \ref{lem:brimoresults} (3) gives that
		\begin{equation*}
			\xi_{-2k}\left(\mathfrak F_{-2k}(\mathbb G(\tau))\right) = -\xi_{-2k}(\mathbb G(\tau)), \quad \mathcal{D}^{2k+1}(\mathfrak F_{-2k}(\mathbb G(\tau))) = -\mathcal{D}^{2k+1}(\mathbb G(\tau)).
		\end{equation*}
According to \cite[Theorem 1.1 (ii)]{mo1} (or \cite[Proposition 4.1]{brimo1}), $g_{k+1,D}$ decays towards $i\infty$ like a cusp form. Thus, both terms in $\mathbb G$ decay towards $i\infty$ as well. Hence, we infer that
		\begin{equation*}
			\mathfrak F_{-2k}(\mathbb G(\tau))  = -\mathbb G(\tau) .
		\end{equation*}
Letting $k \mapsto k+1$ in \eqref{eq:claimPart1} and using the splitting in Lemma \ref{lem:brimoresults} (3) completes the proof.
\end{proof}


\begin{thebibliography}{99}
	\bibitem{nist} R. Boisvert, C. Clark, D. Lozier, and F. Olver, {\it NIST Handbook of mathematical functions}, Cambridge University Press (2010).
		\bibitem{BFOR} K. Bringmann, A. Folsom, K. Ono, and L. Rolen, {\it Harmonic Maass forms and mock modular forms: theory and applications}, AMS colloquium series (2017).
		\bibitem{BKK} K. Bringmann, B. Kane, and W. Kohnen, {\it Locally harmonic Maass forms and the kernel of the Shintani lift}, International Mathematics Research Notices {\bf1} (2015), 3185--3224.
		\bibitem{bkr} K. Bringmann, B. Kane, and R. Rhoades, {\it Duality and differential operators for harmonic Maass forms}, in: {\it From Fourier analysis and number theory to Radon transforms and geometry} 85--106, Dev. Math. {\bf28}, Springer, New York.
		\bibitem{brimo1} K. Bringmann and A. Mono, {\it A modular framework for functions of Knopp and indefinite binary quadratic forms}, arXiv:2208.01451.
		\bibitem{brufu} J. Bruinier and J. Funke, {\it On two geometric theta lifts}, Duke Math. J. {\bf 125} (2004), no.~1, 45--90.
		\bibitem{bronrh} J. Bruinier, K. Ono, and R. Rhoades, {\it Differential operators for harmonic weak Maass forms and the vanishing of Hecke eigenvalues}, Math. Ann. {\bf 342} (2008), no.~3, 673--693.
		\bibitem{123} J. Bruinier, G. van der Geer, G. Harder, and D. Zagier, {\it The 1-2-3 of modular forms}, Universitext, Springer, Berlin, 2008.
		\bibitem{B} D. Bump, {\it Automorphic forms and representations}, Cambridge Studies in Advanced Mathematics {\bf55}, Cambridge Univ. Press, Cambridge, 1997.
		\bibitem{egkr} S. Ehlen, P. Guerzhoy, B. Kane, and L. Rolen, {\it Central $L$-values of elliptic curves and local polynomials}, Proc. Lond. Math. Soc. (3) {\bf 120} (2020), no.~5, 742--769.
		\bibitem{eichler} M. Eichler, {\it Eine Verallgemeinerung der Abelschen Integrale}, Math. Z. {\bf 67} (1957), 267--298.
		\bibitem{IOS} \"O. Imamo\=glu and C. O'Sullivan, {\it Parabolic, hyperbolic and elliptic Poincaré series}, Acta Arith. {\bf139} (2009), no. 3, 199--228.
		\bibitem{kohnen85} W. Kohnen, {\it Fourier coeffcients of modular forms of half-integral weight}, Math. Ann. {\bf271} (1985), no. 2, 237--268.
		\bibitem{katok} S. Katok, Closed geodesics, periods and arithmetic of modular forms, Invent. Math. {\bf 80} (1985), no.~3, 469--480.
		\bibitem{knopp} {M. Knopp}, \textit{Modular integrals and their Mellin transforms}, (1990), {327--342}, \textit{Analytic number theory}, {Allerton Park, IL}, {(1989)}, {Progr. Math.} \textbf{85}, {Birkh\"{a}user Boston, Boston, MA}.
		\bibitem{koza81} W. Kohnen and D. Zagier, {\it Values of L-series of modular forms at the center of the critical strip}, Invent. Math. {\bf64} (1981), no. 2, 175--198.
		\bibitem{koza84} W. Kohnen and D. Zagier, {\it Modular forms with rational periods}, Modular forms (Durham, 1983), Ellis Horwood Ser. Math. Appl.: Statist. Oper. Res., Horwood, Chichester, 1984, pp. 197--249.
		\bibitem{mmrw} J. Males, A. Mono, L. Rolen, and I. Wagner, {\it Central $L$-values of newforms and local polynomials}, arXiv:2306.15519.
		\bibitem{mo1} A. Mono, {\it Locally harmonic Maass forms of positive even weight}, Israel J. Math. {\bf 261} (2024), no.~2, 671--694.
		\bibitem{zagier75} D. Zagier, {\it Modular forms associated to real quadratic fields}, Invent. Math. {\bf30} (1975), 1--46.
		\bibitem{parson} L. Parson, {\it Modular integrals and indefinite binary quadratic forms}, {\it A tribute to Emil Grosswald: number theory and related analysis}, Contemp. Math. {\bf143}, Amer. Math. Soc., Providence, RI (1993), 513--523.
		\bibitem{zagier77} D. Zagier, {\it Modular forms whose Fourier coefficients involve zeta-functions of quadratic fields}, {\it Modular functions of one variable, VI}, Proc. Second Internat. Conf., Univ. Bonn, Bonn (1976), Springer, Berlin, (1977) 105--169. Lecture Notes in Math., Vol. 627.
	\end{thebibliography}
\end{document}